\documentclass[secthm,seceqn,amsthm,ussrhead,reqno, 12pt]{amsart}
\usepackage[utf8]{inputenc}
\usepackage[english]{babel}
\usepackage[symbol]{footmisc}
\usepackage{amssymb,amsmath,amsthm,amsfonts,xcolor,enumerate,hyperref,comment,longtable,cleveref}

\usepackage{times}
\usepackage{cite}
\usepackage{pdflscape}
\usepackage{ulem}
\usepackage[mathcal]{euscript}
\usepackage{tikz}
\usepackage{hyperref}
\usepackage{cancel}
\usepackage{stmaryrd}

\newtheorem{thrm}{Theorem} 

\newtheorem{lem}[thrm]{Lemma}
\newtheorem{prop}[thrm]{Proposition}

\theoremstyle{definition}
\newtheorem{defn}[thrm]{Definition}

\newtheorem{rem}[thrm]{Remark}

\crefrangeformat{equation}{#3(#1)#4--#5(#2)#6}

\crefname{thrm}{Theorem}{Theorems}
\crefname{lem}{Lemma}{Lemmas}
\crefname{cor}{Corollary}{Corollaries}
\crefname{prop}{Proposition}{Propositions}
\crefname{defn}{Definition}{Definitions}
\crefname{exm}{Example}{Examples}
\crefname{rem}{Remark}{Remarks}
\crefname{section}{Section}{Sections}
\crefname{equation}{\unskip}{\unskip}
\crefname{enumi}{\unskip}{\unskip}

\DeclareMathOperator{\Ann}{Ann}

\renewcommand{\iff}{\Leftrightarrow}

\newcommand{\gen}[1]{\langle #1\rangle}

\newcommand{\af}{\alpha}
\newcommand{\bt}{\beta}
\newcommand{\gm}{\gamma}
\newcommand{\dl}{\delta}

\newcommand{\Dl}{\Delta}

\newcommand{\vf}{\varphi}

\newcommand{\LL}{\mathfrak{L}}

\newcommand{\sst}{\subseteq}

\newcommand{\id}{\mathrm{id}}

\newcommand{\ch}{\mathrm{char}}
\newcommand{\tr}{\mathrm{tr}}

\newcommand{\m}{{}^{-1}}

\begin{document}

	\noindent{\Large  
		Transposed Poisson structures on the Lie algebra of \\
  upper triangular matrices}
  \footnote{		The work was supported by  FCT   UIDB/MAT/00212/2020, UIDP/MAT/00212/2020, 2022.02474.PTDC and by CMUP, member of LASI, which is financed by national funds through FCT --- Fundação para a Ciência e a Tecnologia, I.P., under the project with reference UIDB/00144/2020.} 
	\footnote{Corresponding author: kaygorodov.ivan@gmail.com}
	
	\bigskip

	{\bf
		Ivan Kaygorodov\footnote{CMA-UBI, Universidade da Beira Interior, Covilh\~{a}, Portugal; \    kaygorodov.ivan@gmail.com}	\&
		Mykola Khrypchenko\footnote{ Departamento de Matem\'atica, Universidade Federal de Santa Catarina,     Brazil; \and CMUP, Departamento de Matemática, Faculdade de Ciências, Universidade do Porto,			Rua do Campo Alegre s/n, 4169--007 Porto, Portugal\ nskhripchenko@gmail.com}
	}
	\

	\bigskip
 
	\ 
	
	\noindent {\bf Abstract:} {\it 	
		We describe transposed Poisson structures on  the upper triangular matrix Lie algebra $T_n(F)$, $n>1$, over a field $F$ of characteristic zero. We prove that, for $n>2$, any such structure is either of Poisson type or the orthogonal sum of a fixed non-Poisson structure with a structure of Poisson type, and for $n=2$, there is one more class of transposed Poisson structures on $T_n(F)$.
  We also show that, up to isomorphism, the full matrix Lie algebra $M_n(F)$ admits only one non-trivial transposed Poisson structure, and it is of Poisson type.

	}

	\
	

	

	\section*{Introduction} 
	 Since their origin in the 1970s in Poisson geometry, Poisson algebras have appeared in several areas of mathematics and physics, such as algebraic geometry, operads, quantization theory, quantum groups, and classical and quantum mechanics. One of the natural tasks in the theory of Poisson algebras is the description of all such algebras with fixed Lie or associative part~\cite{YYZ07,jawo,kk21}.
  
	Recently, Bai, Bai, Guo, and Wu~\cite{bai20} have introduced a dual notion of the Poisson algebra, called a \textit{transposed Poisson algebra}, by exchanging the roles of the two multiplications in the Leibniz rule defining a Poisson algebra. A transposed Poisson algebra defined this way not only shares some properties of a Poisson algebra, such as the closedness under tensor products and the Koszul self-duality as an operad, but also admits a rich class of identities. It is important to note that a transposed Poisson algebra naturally arises from a Novikov-Poisson algebra by taking the commutator Lie algebra of its Novikov part.
 
	Thanks to \cite{bfk22}, 
	any unital transposed Poisson algebra is
	a particular case of a ``contact bracket'' algebra 
	and a quasi-Poisson algebra.
	In a recent paper by Ferreira, Kaygorodov, and  Lopatkin
	a relation between $\frac{1}{2}$-derivations of Lie algebras and 
	transposed Poisson algebras has been established \cite{FKL}. 	These ideas were used to describe all transposed Poisson structures 
	on  Witt and Virasoro algebras in  \cite{FKL};
	on   twisted Heisenberg-Virasoro,   Schr\"odinger-Virasoro  and  
	extended Schr\"odinger-Virasoro algebras in \cite{yh21};
	on   oscillator Lie algebras in  \cite{bfk22};
        on Schrodinger algebra in $(n+1)$-dimensional space-time in \cite{ytk};
        on Witt type Lie algebras in \cite{kk23};
 on generalized Witt algebras in \cite{kkg23} and Block Lie algebras in \cite{kk22,kkg23}.
	Any complex finite-dimensional solvable Lie algebra was proved to admit a non-trivial transposed Poisson structure \cite{klv22}.
	The algebraic and geometric classification of $3$-dimensional transposed Poisson algebras was given in \cite{bfk23}.	
 For the list of actual open questions on transposed Poisson algebras, see \cite{bfk22}.

In this paper we describe transposed Poisson structures on the upper triangular matrix Lie algebra $T_n(F)$ over a field $F$ of characteristic zero. To this end, we first characterize $\frac 12$-derivations of $T_n(F)$ in \cref{descr-Dl(T_n(F))}. Then the case $n>2$ is solved in \cref{descr-TP-on-T_n(F)}, and the case $n=2$ is treated separately in \cref{descr-TP-on-T_2(F)}. Namely, we prove in \cref{descr-TP-on-T_n(F)} that any transposed Poisson structure on $T_n(F)$, $n>2$, is either of Poisson type or the orthogonal sum of the fixed non-Poisson structure 
\begin{align*}
			e_{11}\cdot e_{11}=-e_{11}\cdot e_{nn}=e_{nn}\cdot e_{nn}=e_{1n}
\end{align*}
with a structure of Poisson type. If $n=2$, then there appears one more separate class of transposed Poisson structures on $T_n(F)$ consisting of the (non-orthogonal) sums of a structure of the family
\begin{align*}
    e_{11}\cdot e_{11}=ce_{11},
    e_{11}\cdot e_{12}=-e_{12}\cdot e_{22}=ce_{12},
    e_{11}\cdot e_{22}=-e_{22}\cdot e_{22}=ce_{22},\ c\ne 0,
\end{align*}
with a structure of Poisson type.
As a complementary result, we prove in \cref{descr-TP-on-M_n(F)} that the full matrix Lie algebra $M_n(F)$ admits only one non-trivial transposed Poisson structure $e_{ii}\cdot e_{jj}=\dl$, $1\le i,j\le n$, and it is of Poisson type. In fact, is isomorphic to the extension by zero of the product $\dl\cdot \dl=\dl$ as observed in \cref{ext-dl.dl=dl}.

	\section{Definitions and preliminaries}\label{prelim}
	
	All the algebras below will be over a field $F$ of characteristic zero and all the linear maps will be $F$-linear, unless otherwise stated. The notation $\gen{S}$ means the $F$-subspace generated by $S$.

	\begin{defn}\label{tpa}
		Let ${\mathfrak L}$ be a vector space equipped with two nonzero bilinear operations $\cdot$ and $[\cdot,\cdot].$
		The triple $({\mathfrak L},\cdot,[\cdot,\cdot])$ is called a \textit{transposed Poisson algebra} if $({\mathfrak L},\cdot)$ is a commutative associative algebra and
		$({\mathfrak L},[\cdot,\cdot])$ is a Lie algebra that satisfies the following compatibility condition
		\begin{align}\label{Trans-Leibniz}
			2z\cdot [x,y]=[z\cdot x,y]+[x,z\cdot y].
		\end{align}
	\end{defn}
	
	Transposed Poisson algebras were first introduced in a paper by Bai, Bai, Guo, and Wu \cite{bai20}.
	
	\begin{defn}\label{tp-structures}
		Let $({\mathfrak L},[\cdot,\cdot])$ be a Lie algebra. A \textit{transposed Poisson algebra structure} on $({\mathfrak L},[\cdot,\cdot])$ is a commutative associative multiplication $\cdot$ on $\mathfrak L$ which makes $({\mathfrak L},\cdot,[\cdot,\cdot])$ a transposed Poisson algebra.
	\end{defn}

	\begin{defn}\label{12der}
		Let $({\mathfrak L}, [\cdot,\cdot])$ be an algebra and $\varphi:\mathfrak L\to\mathfrak L$ a linear map.
		Then $\varphi$ is a \textit{$\frac{1}{2}$-derivation} if it satisfies
		\begin{align}\label{vf(xy)=half(vf(x)y+xvf(y))}
			\varphi \big([x,y]\big)= \frac{1}{2} \big([\varphi(x),y]+ [x, \varphi(y)] \big).
		\end{align}
	\end{defn}
	Observe that $\frac{1}{2}$-derivations are a particular case of $\delta$-derivations introduced by Filippov in \cite{fil1}. The space of all $\frac{1}{2}$-derivations of an algebra $\mathfrak L$ will be denoted by $\Dl(\mathfrak L).$ It is easy to see from \cref{vf(xy)=half(vf(x)y+xvf(y))} that $[\LL,\LL]$ and $\Ann(\LL)$ are invariant under any $\frac 12$-derivation of $\LL$.

	\cref{tpa,12der} immediately imply the following key Lemma.
	\begin{lem}\label{glavlem}
		Let $({\mathfrak L},[\cdot,\cdot])$ be a Lie algebra and $\cdot$ a new binary (bilinear) operation on ${\mathfrak L}$. Then $({\mathfrak L},\cdot,[\cdot,\cdot])$ is a transposed Poisson algebra 
		if and only if $\cdot$ is commutative and associative and for every $z\in{\mathfrak L}$ the multiplication by $z$ in $({\mathfrak L},\cdot)$ is a $\frac{1}{2}$-derivation of $({\mathfrak L}, [\cdot,\cdot]).$
	\end{lem}
	
	The basic example of a $\frac{1}{2}$-derivation is the multiplication by a field element.
	Such $\frac{1}{2}$-derivations will be called \textit{trivial}. 
	
	\begin{thrm}\label{princth}
		Let ${\mathfrak L}$ be a Lie algebra without non-trivial $\frac{1}{2}$-derivations.
		Then all transposed Poisson algebra structures on ${\mathfrak L}$ are trivial.
	\end{thrm}

 Another well-known class of $\frac{1}{2}$-derivations of $(\LL,[\cdot,\cdot])$ is formed by linear maps $\LL\to \Ann(\LL)$ annihilating $[\LL,\LL]$. If $(\LL,[\cdot,\cdot])$ is a Lie algebra, such $\frac{1}{2}$-derivations of $\LL$ correspond to the following transposed Poisson structures on $\LL$. Denote by $Z(\LL)$ the \textit{center} of $\LL$ and fix a complement $V$ of $[\LL,\LL]$ in $\LL$. Then any commutative associative product $*:V\times V\to Z(\LL)$ defines a transposed Poisson algebra structure $\cdot$ on $\LL$ by means of
	\begin{align}\label{(a_1+a_2)-times-(b_1+b_2)}
		(a_1+a_2)\cdot(b_1+b_2)=a_1*b_1,
	\end{align}
	where $a_1,b_1\in V$ and $a_2,b_2\in [\LL,\LL]$. Indeed, the right-hand side of \cref{Trans-Leibniz} is zero, because $z\cdot x,z\cdot y\in Z(\LL)$, and the left-hand side of \cref{Trans-Leibniz} is zero by \cref{(a_1+a_2)-times-(b_1+b_2)}, because $[x,y]\in[\LL,\LL]$. We say that $\cdot$ is \textit{the extension by zero} of $*$. Observe that $\cdot$ is, at the same time, a usual Poisson structure on $(\LL,[\cdot,\cdot])$. Thus, such transposed Poisson structures are said to be \textit{of Poisson type}.

 Given two transposed Poisson structures $\cdot_1$ and $\cdot_2$ on $(\LL,[\cdot,\cdot])$, their \textit{sum} $*$ defined by 
 \begin{align*}
     a*b=a\cdot_1 b+a\cdot_2 b
 \end{align*}
 is clearly commutative and satisfies \cref{Trans-Leibniz}. In general, $*$ may be non-associative, but it is associative, if
 \begin{align*}
     \LL\cdot_1\LL\sst\Ann(\LL,\cdot_2)\text{ and }\LL\cdot_2\LL\sst\Ann(\LL,\cdot_1).
 \end{align*}
 In this case we say that $\cdot_1$ and $\cdot_2$ are \textit{orthogonal}, and $*$ is the \textit{orthogonal} sum of $\cdot_1$ and $\cdot_2$. 

	
	 Let $\cdot$ be a transposed Poisson algebra structure on a Lie algebra $({\mathfrak L}, [\cdot,\cdot])$. 
	 Then any automorphism $\phi$ of $({\mathfrak L}, [\cdot,\cdot])$ induces the transposed Poisson algebra structure $*$ on $({\mathfrak L}, [\cdot,\cdot])$ given by
	 \begin{align*}
		 x*y=\phi\big(\phi^{-1}(x)\cdot\phi^{-1}(x)\big),\ \ x,y\in{\mathfrak L}.
		 \end{align*}
	 Clearly, $\phi$ is an isomorphism of transposed Poisson algebras $({\mathfrak L},\cdot,[\cdot,\cdot])$ and $({\mathfrak L},*,[\cdot,\cdot])$.

	\section{Transposed Poisson structures on the upper triangular matrix Lie algebra}\label{sec-TP-on-T_n}

\subsection{Upper triangular matrix algebra}\label{sec-gen-T-n}
	
Let $n$ be a positive integer. Denote by $T_n(F)$ the algebra of upper triangular $n\times n$ matrices over $F$. The usual matrix product on $T_n(F)$ will be denoted by $\cdot$ or just by the concatenation, and the commutator product by $[a,b]=ab-ba$. Following the terminology of incidence algebras, we denote the identity matrix by $\dl$ and, given $a\in T_n(F)$ and $1\le i,j\le n$, we write $a(i,j)$ for the $(i,j)$-entry of $a$. The algebra $T_n(F)$ has the natural basis formed by the matrix units $e_{ij}$, $1\le i\le j\le n$, where $e_{ij}(k,l)=\dl(i,k)\dl(j,l)$. It is well-known that
\begin{align*}
    Z(T_n(F))=\gen{\dl}\text{ and }[T_n(F),T_n(F)]=\gen{e_{ij}\mid 1\le i<j\le n}\text{ for all }n\ge 1.     
\end{align*}
Moreover, 
\begin{align}\label{Z([T_n(F)_T_n(F)])=<e_1n>}
    Z([T_n(F),T_n(F)])=\gen{e_{1n}}\text{ for all }n>1.
\end{align}
Clearly, for all $a\in T_n(F)$ and $1\le i\le j\le n$  we have
\begin{align*}
	e_{ij}a=\sum_{j\le k}a(j,k)e_{ik},\ \ ae_{ij}=\sum_{l\le i}a(l,i)e_{lj},
\end{align*}
so
\begin{align*}
	e_{ij}ae_{kl}=
	\begin{cases}
		a(j,k)e_{il}, & j\le k,\\
		0, & j>k.
	\end{cases}
\end{align*}
These equalities will be used numerous times throughout the text without any reference.

\subsection{$\frac{1}{2}$-derivations of the upper triangular matrix Lie algebra}
 
We denote by $\Dl(T_n(F))$ the space of $\frac 12$-derivations of the Lie algebra $(T_n(F),[\cdot,\cdot])$.

\begin{lem}\label{basic-properties-of-vf(e_ij)}
	Let $\vf\in\Dl(T_n(F))$. Then for all $1\le i<j\le n$ we have
	\begin{enumerate}
		\item\label{vf(e_ij)e_ii} $\vf(e_{ij})e_{ii}=0;$
		\item\label{e_ij.vf(e_ii)} $e_{ij}\vf(e_{ii})=\vf(e_{ii})(j,j)e_{ij};$
		\item\label{vf(e_ii)(i_j)=-vf(e_jj)(i_j)} $\vf(e_{ii})(i,j)=-\vf(e_{jj})(i,j);$
		\item\label{vf(e_ii)e_ij=vf(e_ii)(i_i)e_ij} $\vf(e_{ii})e_{ij}=\vf(e_{ii})(i,i)e_{ij}$.
	\end{enumerate}
\end{lem}
\begin{proof}
	\textit{\cref{vf(e_ij)e_ii}}. Since $e_{ij}=[e_{ij},e_{jj}]$, then
	\begin{align}\label{2vf(e_ij)=[vf(e_ij)_e_jj]+[e_ij_vf(e_jj)]}
		2\vf(e_{ij})=[\vf(e_{ij}),e_{jj}]+[e_{ij},\vf(e_{jj})]=\vf(e_{ij})e_{jj}-e_{jj}\vf(e_{ij})+e_{ij}\vf(e_{jj})-\vf(e_{jj})e_{ij}.
	\end{align}
	Multiplying \cref{2vf(e_ij)=[vf(e_ij)_e_jj]+[e_ij_vf(e_jj)]} by $e_{ii}$ on the right (under $\cdot$), we obtain the desired equality.
	
	\textit{\cref{e_ij.vf(e_ii)}}. Apply $\vf$ to $[e_{ii},e_{jj}]=0$:
	\begin{align}\label{0=vf(e_ii)e_jj-e_jj.vf(e_ii)+e_ii.vf(e_jj)-vf(e_jj)e_ii}
		0=\vf([e_{ii},e_{jj}])=\vf(e_{ii})e_{jj}-e_{jj}\vf(e_{ii})+e_{ii}\vf(e_{jj})-\vf(e_{jj})e_{ii}.
	\end{align} 
	If there exists $k>j$, then multiplying \cref{0=vf(e_ii)e_jj-e_jj.vf(e_ii)+e_ii.vf(e_jj)-vf(e_jj)e_ii} by $e_{jj}$ on the left and by $e_{kk}$ on the right, we obtain
	\begin{align}\label{vf(e_ii)(j_k)=0-for-i-ne-j_k}
		\vf(e_{ii})(j,k)=0\text{ for }j<k.
	\end{align}
	It follows that
	\begin{align*}
		e_{ij}\vf(e_{ii})=\sum_{j\le k}\vf(e_{ii})(j,k)e_{ik}=\vf(e_{ii})(j,j)e_{ij},
	\end{align*}
	as needed.
	
	\textit{\cref{vf(e_ii)(i_j)=-vf(e_jj)(i_j)}}. This follows by multiplying \cref{0=vf(e_ii)e_jj-e_jj.vf(e_ii)+e_ii.vf(e_jj)-vf(e_jj)e_ii} by $e_{ii}$ on the left and by $e_{jj}$ on the right.
	
	\textit{\cref{vf(e_ii)e_ij=vf(e_ii)(i_i)e_ij}}. Applying $\vf$ to $e_{ij}=[e_{ii},e_{ij}]$, we have
	\begin{align}\label{2vf(e_ij)=[vf(e_ii)_e_ij]+[e_ii_vf(e_ij)]}
		2\vf(e_{ij})=[\vf(e_{ii}),e_{ij}]+[e_{ii},\vf(e_{ij})]=\vf(e_{ii})e_{ij}-e_{ij}\vf(e_{ii})+e_{ii}\vf(e_{ij})-\vf(e_{ij})e_{ii}.
	\end{align}
	If there exists $l<i$, then the multiplication  of \cref{2vf(e_ij)=[vf(e_ii)_e_ij]+[e_ii_vf(e_ij)]} by $e_{ll}$ on the left and by $e_{jj}$ on the right gives
	\begin{align}\label{2.vf(e_ij)(l_j)=vf(e_ii)(l_i)}
		2\vf(e_{ij})(l,j)=\vf(e_{ii})(l,i).
	\end{align}
	On the other hand, $[e_{ll},e_{ij}]=0$, so
	\begin{align*}
		0=\vf(e_{ll})e_{ij}-e_{ij}\vf(e_{ll})+e_{ll}\vf(e_{ij})-\vf(e_{ij})e_{ll}.
	\end{align*}
	Multiplying this by $e_{ll}$ on the left and by $e_{jj}$ on the right, we come to
	\begin{align}\label{vf(e_ij)(l_j)=-vf(e_ll)(l_i)}
		\vf(e_{ij})(l,j)=-\vf(e_{ll})(l,i).
	\end{align}
	The latter is $\vf(e_{ii})(l,i)$ by item \cref{vf(e_ii)(i_j)=-vf(e_jj)(i_j)}. Thus, \cref{2.vf(e_ij)(l_j)=vf(e_ii)(l_i),vf(e_ij)(l_j)=-vf(e_ll)(l_i)} result in 
	\begin{align}\label{vf(e_ii)(l_i)=0}
		\vf(e_{ii})(l,i)=0\text{ for }l<i.
	\end{align}
	Consequently,
	\begin{align*}
		\vf(e_{ii})e_{ij}=\sum_{l\le i}\vf(e_{ii})(l,i)e_{ik}=\vf(e_{ii})(i,i)e_{ij},
	\end{align*}
	as desired.
\end{proof}

\begin{lem}\label{vf(e_ij)-in-<e_ij>}
	Let $\vf\in\Dl(T_n(F))$. Then for all $1\le i<j\le n$ we have
	\begin{align}\label{vf(e_ij)=(vf(e_ii)(i_i)-vf(e_ii)(j_j))e_ij}
		\vf(e_{ij})=(\vf(e_{ii})(i,i)-\vf(e_{ii})(j,j))e_{ij}=(\vf(e_{jj})(j,j)-\vf(e_{jj})(i,i))e_{ij}.
	\end{align}
\end{lem}
\begin{proof}
	By \cref{2vf(e_ij)=[vf(e_ii)_e_ij]+[e_ii_vf(e_ij)]} and \cref{vf(e_ij)e_ii,e_ij.vf(e_ii),vf(e_ii)e_ij=vf(e_ii)(i_i)e_ij} of \cref{basic-properties-of-vf(e_ij)} we have
	\begin{align}\label{2vf(e_ij)=(vf(e_ii)(i_i)-vf(e_ii)(j_j))e_ij+e_ii.vf(e_ij)}
		2\vf(e_{ij})=(\vf(e_{ii})(i,i)-\vf(e_{ii})(j,j))e_{ij}+e_{ii}\vf(e_{ij}).
	\end{align}
	Multiplying \cref{2vf(e_ij)=(vf(e_ii)(i_i)-vf(e_ii)(j_j))e_ij+e_ii.vf(e_ij)} by $e_{ii}$ on the left, we get
	\begin{align*}
		2e_{ii}\vf(e_{ij})=(\vf(e_{ii})(i,i)-\vf(e_{ii})(j,j))e_{ij}+e_{ii}\vf(e_{ij}).
	\end{align*}
	whence
	\begin{align*}
		e_{ii}\vf(e_{ij})=(\vf(e_{ii})(i,i)-\vf(e_{ii})(j,j))e_{ij}.
	\end{align*}
	Substituting this into \cref{2vf(e_ij)=(vf(e_ii)(i_i)-vf(e_ii)(j_j))e_ij+e_ii.vf(e_ij)} and dividing by $2$ (recall that $\ch(F)=0$), we prove the first equality of \cref{vf(e_ij)=(vf(e_ii)(i_i)-vf(e_ii)(j_j))e_ij}. The second one is obtained from \cref{2vf(e_ij)=[vf(e_ij)_e_jj]+[e_ij_vf(e_jj)]} using $\vf(e_{ij})\in\gen{e_{ij}}$ (which holds by the first equality of \cref{vf(e_ij)=(vf(e_ii)(i_i)-vf(e_ii)(j_j))e_ij}).
\end{proof}	
	
\begin{lem}\label{vf(e_ij)(i_j)=const}
	Let $\vf\in\Dl(T_n(F))$. 
	\begin{enumerate}
		\item\label{vf(e_ii)(j_j)=vf(e_ii)(k_k)} For all $1\le i\le n$ and $1\le j\le k\le n$ with $i\not\in\{j,k\}$ we have $\vf(e_{ii})(j,j)=\vf(e_{ii})(k,k)$.
		\item\label{vf(e_ij)(i_j)=vf(e_ij)(1_n)} For all $1\le i<j\le n$ we have $\vf(e_{ij})(i,j)=\vf(e_{1n})(1,n)$.
		\item\label{vf(e_ii)(j_k)=0} For all $1\le i\le n$ and $1\le j<k\le n$ we have $\vf(e_{ii})(j,k)=0$, unless $i\in\{1,n\}$ and $(j,k)=(1,n)$.
	\end{enumerate}
\end{lem}
\begin{proof}
	\textit{\cref{vf(e_ii)(j_j)=vf(e_ii)(k_k)}}. Applying $\vf$ to $[e_{ii},e_{jk}]=0$, we get
	\begin{align*}
		0=\vf(e_{ii})e_{jk}-e_{jk}\vf(e_{ii})+e_{ii}\vf(e_{jk})-\vf(e_{jk})e_{ii}.
	\end{align*}
Then the desired equality is obtained by the left multiplication by $e_{jj}$ and right multiplication by $e_{kk}$.

\textit{\cref{vf(e_ij)(i_j)=vf(e_ij)(1_n)}}. Since $i<j$, then $i\ne n$, so $i\not\in\{j,n\}$ and $n\not\in\{1,i\}$. We use item \cref{vf(e_ii)(j_j)=vf(e_ii)(k_k)} and \cref{vf(e_ij)=(vf(e_ii)(i_i)-vf(e_ii)(j_j))e_ij}:
\begin{align*}
	\vf(e_{ij})(i,j)&=\vf(e_{ii})(i,i)-\vf(e_{ii})(j,j)=\vf(e_{ii})(i,i)-\vf(e_{ii})(n,n)\\
	&=\vf(e_{nn})(n,n)-\vf(e_{nn})(i,i)=\vf(e_{nn})(n,n)-\vf(e_{nn})(1,1)=\vf(e_{1n})(1,n).
\end{align*}

\textit{\cref{vf(e_ii)(j_k)=0}}. We have already seen in \cref{vf(e_ii)(j_k)=0-for-i-ne-j_k} that $\vf(e_{ii})(j,k)=0$ for $j<k$ and $i\not\in\{j,k\}$ (although it is required that $i<j$ in the statement of \cref{basic-properties-of-vf(e_ij)}, the proof of \cref{vf(e_ii)(j_k)=0-for-i-ne-j_k} uses only $j<k$ and $i\not\in\{j,k\}$). Moreover, $\vf(e_{ii})(l,i)=0$ for all $l<i$ by \cref{vf(e_ii)(l_i)=0} under the assumption that there exists $j>i$ (i.e., $i<n$). Similarly, \cref{2vf(e_ij)=[vf(e_ij)_e_jj]+[e_ij_vf(e_jj)],vf(e_ij)-in-<e_ij>} imply that $\vf(e_{jj})(j,k)=2\vf(e_{ij})(j,k)=0$ for all $i<j<k$, so $\vf(e_{ii})(i,k)=0$ for all $1<i<k$. It remains to prove that
\begin{align*}
	\vf(e_{11})(1,k)=\vf(e_{nn})(j,n)=0\text{ for }1<j,k<n.
\end{align*}
But \cref{basic-properties-of-vf(e_ij)}\cref{vf(e_ii)(i_j)=-vf(e_jj)(i_j)} shows that $\vf(e_{11})(1,k)=-\vf(e_{kk})(1,k)$, which is proved to be $0$ for $1<k<n$. Similarly, $\vf(e_{nn})(j,n)=-\vf(e_{jj})(j,n)=0$ for $1<j<n$.
\end{proof}

\begin{lem}\label{af-is-halfder}
	Let $n>1$. Then the linear map $\af:T_n(F)\to T_n(F)$ given by
	\begin{align*}
		\af(e_{ij})=
		\begin{cases}
			e_{1n}, & (i,j)=(1,1),\\
			-e_{1n}, & (i,j)=(n,n),\\
			0, & (i,j)\not\in\{(1,1),(n,n)\},
		\end{cases}
	\end{align*}
is a $\frac 12$-derivation of $T_n(F)$.
\end{lem}
\begin{proof}
	We are going to prove that $\vf=\af$ satisfies \cref{vf(xy)=half(vf(x)y+xvf(y))} for $x=e_{ij}$ and $y=e_{kl}$. Since $\af$ annihilates the commutators of $T_n(F)$, the left-hand side of \cref{vf(xy)=half(vf(x)y+xvf(y))} is always zero. In view of the anti-commutativity of $[\cdot,\cdot]$, we have to deal only with the following $2$ cases.
	
	\textit{Case 1}. $(i,j)=(1,1)$. Then 
	\begin{align}\label{[af(e_11)_e_kl]+[e_11_af(e_kl)]}
		[\af(e_{ij}),e_{kl}]+ [e_{ij}, \af(e_{kl})]=[e_{1n},e_{kl}]+ [e_{11}, \af(e_{kl})].
	\end{align}
Both summands on the right-hand side of \cref{[af(e_11)_e_kl]+[e_11_af(e_kl)]} are zero, unless $(k,l)\in\{(1,1),(n,n)\}$. If $(k,l)=(1,1)$, then the right-hand side of \cref{[af(e_11)_e_kl]+[e_11_af(e_kl)]} is zero due to the anti-commutativity of $[\cdot,\cdot]$. And if $(k,l)=(n,n)$, then $[e_{1n},e_{kl}]+ [e_{11}, \af(e_{kl})]=e_{1n}-e_{1n}=0$.

\textit{Case 2}. $(i,j)=(n,n)$. Then 
\begin{align*}
	[\af(e_{ij}),e_{kl}]+ [e_{ij}, \af(e_{kl})]=-[e_{1n},e_{kl}]+ [e_{nn}, \af(e_{kl})].
\end{align*}
We again have only $2$ non-trivial subcases $(k,l)=(1,1)$ and $(k,l)=(n,n)$ that are similar to the corresponding subcases of Case 1.
\end{proof}

We also introduce the linear maps $\bt_i:T_n(F)\to T_n(F)$, $1\le i\le n$, by
\begin{align*}
	\bt_i(e_{jk})=
	\begin{cases}
		\dl, & (j,k)=(i,i),\\
		0, & (j,k)\ne(i,i).
	\end{cases}
\end{align*}
Obviously, $\bt_i$, $1\le i\le n$, constitute a basis of the space of linear maps $T_n(F)\to Z(T_n(F))$ annihilating $[T_n(F),T_n(F)]$. In particular, $\bt_i\in \Dl(T_n(F))$ for all $1\le i\le n$.

\begin{prop}\label{descr-Dl(T_n(F))}
	Let $n>1$. Then $\Dl(T_n(F))=\gen{\id,\af}\oplus\gen{\bt_i\mid 1\le i\le n}$.
\end{prop}
\begin{proof}
	In view of \cref{af-is-halfder} we only need to prove that each $\vf\in\Dl(T_n(F))$ is a linear combination of $\id$, $\af$ and $\bt_i$, $1\le i\le n$. Setting $a=\vf(e_{1n})(1,n)$, we have 
	\begin{align}\label{vf(e_ij)=ae_ij}
		\vf(e_{ij})=ae_{ij}\text{ for all }1\le i<j\le n
	\end{align} 
by \cref{vf(e_ij)-in-<e_ij>} and \cref{vf(e_ij)(i_j)=const}\cref{vf(e_ij)(i_j)=vf(e_ij)(1_n)}. Then
	\begin{align*}
		\vf(e_{ii})(i,i)-\vf(e_{ii})(j,j)=a=\vf(e_{ii})(i,i)-\vf(e_{ii})(k,k)
	\end{align*}
for all $k<i<j$ by \cref{vf(e_ij)=(vf(e_ii)(i_i)-vf(e_ii)(j_j))e_ij}. Hence, defining $b_i=\vf(e_{ii})(j,j)$ for some $j\ne i$, we have
	\begin{align}\label{vf(e_ii)(j_j)=a+b_i-or-b_i}
		\vf(e_{ii})(j,j)=
		\begin{cases}
			a+b_i, & j=i,\\
			b_i, & j\ne i.
		\end{cases}
	\end{align}
Finally, let $c=\vf(e_{11})(1,n)$. By \cref{basic-properties-of-vf(e_ij)}\cref{vf(e_ii)(i_j)=-vf(e_jj)(i_j)} and \cref{vf(e_ij)(i_j)=const}\cref{vf(e_ii)(j_k)=0} for all $1\le i\le n$ and $1\le j<k\le n$ we have
\begin{align}\label{vf(e_i)(j_k)=c-or-minus-c-or-0}
	\vf(e_{ii})(j,k)=
	\begin{cases}
		c, & (i,j,k)=(1,1,n),\\
		-c, & (i,j,k)=(n,1,n),\\
		0, & (i,j,k)\not\in\{(1,1,n),(n,1,n)\}.
	\end{cases}
\end{align}
It follows from \cref{vf(e_ij)=ae_ij,vf(e_ii)(j_j)=a+b_i-or-b_i,vf(e_i)(j_k)=c-or-minus-c-or-0} that $\vf=a\cdot\id+c\af+\sum_{i=1}^n b_i\beta_i$.
\end{proof}

\begin{thrm}\label{descr-TP-on-T_n(F)}
	Let $\ch(F)=0$ and $n>2$. Then any transposed Poisson algebra structure on $T_n(F)$ is of one of the following two non-isomorphic forms:
	\begin{enumerate}
		\item\label{cdot-of-Poisson-type} transposed Poisson algebra structure of Poisson type;
		\item\label{cdot-sum-of-Poisson-type-with-fixed} the orthogonal sum of the transposed Poisson algebra structure
		\begin{align}\label{e_11.e_11=e_nn.e_nn=-e_11.e_nn=-e_nn.e_11=e_1n}
			e_{11}\cdot e_{11}=-e_{11}\cdot e_{nn}=e_{nn}\cdot e_{nn}=e_{1n}
		\end{align}
  with a transposed Poisson algebra structure of Poisson type.
	\end{enumerate} 
\end{thrm}
\begin{proof}
	Let $\cdot$ be a transposed Poisson algebra structure on $T_n(F)$. By \cref{descr-Dl(T_n(F)),glavlem} for all $1\le i\le j\le n$ there exist $x_{ij},y_{ij}\in F$ and $\{z_{ij}^k\}_{k=1}^n\sst F$, such that
	\begin{align*}
		e_{ij}\cdot e_{kl}=x_{ij}e_{kl}+y_{ij}\af(e_{kl})+\sum_{s=1}^n z_{ij}^s\bt_s(e_{kl})=
		\begin{cases}
			x_{ij}e_{kl}, & k<l,\\
			x_{ij}e_{kk}+z_{ij}^k\dl, & k=l\not\in\{1,n\},\\
			x_{ij}e_{11}+y_{ij}e_{1n}+z_{ij}^1\dl, & k=l=1,\\
			x_{ij}e_{nn}-y_{ij}e_{1n}+z_{ij}^n\dl, & k=l=n.
		\end{cases}
	\end{align*}

Let $1\le i<j\le n$ and $1<k<n$. Then $e_{kk}\cdot e_{ij}=x_{kk}e_{ij}$, while $e_{ij}\cdot e_{kk}=x_{ij}e_{kk}+z_{ij}^k\dl$, whence
\begin{align}\label{x_kk=x_ij=x_ij=z_ij^k=0}
	x_{kk}=0\text{ and }x_{ij}=z_{ij}^k=0\text{ for }1\le i<j\le n\text{ and }1<k<n.
\end{align}
Now take $1\le i<j\le n$ with $(i,j)\ne (1,n)$. Then $e_{11}\cdot e_{ij}=x_{11}e_{ij}$, while $e_{ij}\cdot e_{11}=y_{ij}e_{1n}+z_{ij}^1\dl$. Similarly, $e_{nn}\cdot e_{ij}=x_{nn}e_{ij}$, while $e_{ij}\cdot e_{nn}=-y_{ij}e_{1n}+z_{ij}^n\dl$. Hence,
\begin{align}\label{x_11=x_nn=0}
	x_{11}=x_{nn}=0\text{ and }y_{ij}=z_{ij}^1=z_{ij}^n=0\text{ for }1\le i<j\le n\text{ with }(i,j)\ne (1,n).
\end{align}
Considering the products $e_{11}\cdot e_{1n}=0$, $e_{1n}\cdot e_{11}=y_{1n}e_{1n}+z_{1n}^1\dl$, $e_{nn}\cdot e_{1n}=0$ and $e_{1n}\cdot e_{nn}=-y_{1n}e_{1n}+z_{1n}^n\dl$, we have
\begin{align}\label{y_1n=z_1n^1=z_1n^n=0}
	y_{1n}=z_{1n}^1=z_{1n}^n=0.
\end{align}
Now, for any $1<i<n$ it follows from $e_{11}\cdot e_{ii}=z_{11}^i\dl$, $e_{ii}\cdot e_{11}=y_{ii}e_{1n}+z_{ii}^1\dl$, $e_{nn}\cdot e_{ii}=z_{nn}^i\dl$ and $e_{ii}\cdot e_{nn}=-y_{ii}e_{1n}+z_{ii}^n\dl$ that
\begin{align}\label{y_ii=0_z_ii^1=z_11^i_z_ii^n=z_nn^i}
	y_{ii}=0,\ z_{ii}^1=z_{11}^i\text{ and }z_{ii}^n=z_{nn}^i\text{ for }1<i<n.
\end{align}
Similarly, for $1<i,j<n$, it follows from $e_{ii}\cdot e_{jj}=z_{ii}^j\dl$ and $e_{jj}\cdot e_{ii}=z_{jj}^i\dl$ that
\begin{align}\label{z_ii^j=z_jj^i}
	z_{ii}^j=z_{jj}^i\text{ for }1<i,j<n.
\end{align}
Finally, $e_{11}\cdot e_{nn}=-y_{11}e_{1n}+z_{11}^n\dl$ and $e_{nn}\cdot e_{11}=y_{nn}e_{1n}+z_{nn}^1\dl$ yield
\begin{align}\label{y_nn=-y_11_z_11^n=z_nn^1}
	y_{nn}=-y_{11}\text{ and }z_{11}^n=z_{nn}^1.
\end{align}
Combining \cref{x_kk=x_ij=x_ij=z_ij^k=0,x_11=x_nn=0,y_1n=z_1n^1=z_1n^n=0,y_ii=0_z_ii^1=z_11^i_z_ii^n=z_nn^i,z_ii^j=z_jj^i,y_nn=-y_11_z_11^n=z_nn^1} and denoting $a_{ij}:=z_{ii}^j$, $b=y_{11}$, we see that the only (possibly) non-zero products $e_{ij}\cdot e_{kl}$ are
\begin{align}
	e_{ii}\cdot e_{jj}&=a_{ij}\dl,\ 1<i,j<n,\label{e_i.e_j=a_ij.dl}\\
	e_{11}\cdot e_{11}&=be_{1n}+a_{11}\dl,\ e_{nn}\cdot e_{nn}=be_{1n}+a_{nn}\dl,\ e_{11}\cdot e_{nn}=e_{nn}\cdot e_{11}=-be_{1n}+a_{1n}\dl,\label{e_1.e_1-e_n.e_n-e_1.e_n}
\end{align}
where $a_{ij}=a_{ji}$ for all $1\le i,j\le n$. 

If $b=0$, then $\cdot$ is of Poisson type, so we are in the case \cref{cdot-of-Poisson-type}. Otherwise, choosing $\phi$ to be the conjugation by $(b\m-1) e_{11}+\dl$, we have 
\begin{align*}
	\phi(e_{ij})=
	\begin{cases}
		e_{ij}, & 1<i\le j\le n\text{ or }i=j=1,\\
		b\m e_{ij}, & 1=i<j\le n.´
	\end{cases}
\end{align*}
It follows that, applying $\phi$, we can replace $b\ne 0$ by $b=1$ in \cref{e_1.e_1-e_n.e_n-e_1.e_n}, and we come to the form \cref{cdot-sum-of-Poisson-type-with-fixed}. 

To prove that the structures \cref{cdot-of-Poisson-type,cdot-sum-of-Poisson-type-with-fixed} are not isomorphic, observe from~\cite{MS} that for any automorphism $\phi$ of $(T_n(F),[\cdot,\cdot])$
\begin{align}
	\text{either } &\phi(e_{ii})\in e_{ii}+(\gen{\dl}\oplus\gen{e_{ij}\mid i<j})\text{ for all }1\le i\le n,\label{phi(e_ii)=e_ii+c.dl}\\
	\text{or } &\phi(e_{ii})\in -e_{n-i+1,n-i+1}+(\gen{\dl}\oplus\gen{e_{ij}\mid i<j})\text{ for all }1\le i\le n.\label{phi(e_ii)=-e_n-i+1_n-i+1+c.dl}
\end{align}
Consider the case \cref{phi(e_ii)=e_ii+c.dl}. Since $e_{ii}\cdot \dl\in\gen{\dl}=Z(T_n(F),[\cdot,\cdot])$ and $e_{ii}\cdot e_{jk}=0$ for all $1\le i\le n$ and $1\le j<k\le n$ by \cref{e_i.e_j=a_ij.dl,e_1.e_1-e_n.e_n-e_1.e_n}, any such $\phi$ leads to the product
\begin{align*}
	e_{ii}*e_{jj}=\phi(\phi\m(e_{ii})\cdot\phi\m(e_{jj}))\in \phi(e_{ii}\cdot e_{jj}+\gen{\dl})=\phi(e_{ii}\cdot e_{jj})+\gen{\dl},
\end{align*}
where $\phi(e_{ii}\cdot e_{jj})$ belongs either to $\gen{\dl}$ (whenever $e_{ii}\cdot e_{jj}\in\gen{\dl}$) or to $\pm b\phi(e_{1n})+\gen{\dl}$ (whenever $e_{ii}\cdot e_{jj}\in \pm be_{1n}+\gen{\dl}$). But $\phi(e_{1n})$ is a non-zero multiple of $e_{1n}$ by \cref{Z([T_n(F)_T_n(F)])=<e_1n>}, so 
\begin{align}\label{e_ii*e_jj-in-<dl>-iff-e_ii.e_jj-in-<dl>}
    e_{ii}*e_{jj}\in\gen{\dl}\iff e_{ii}\cdot e_{jj}\in\gen{\dl}.    
\end{align}
for all $1\le i,j\le n$. This shows that any structure of type \cref{cdot-of-Poisson-type} can be isomorphic only to a structure of type \cref{cdot-of-Poisson-type}. As to the case \cref{phi(e_ii)=-e_n-i+1_n-i+1+c.dl}, we similarly have
\begin{align*}
	e_{ii}*e_{jj}=\phi(\phi\m(e_{ii})\cdot\phi\m(e_{jj}))\in \phi(e_{n-i+1,n-i+1}\cdot e_{n-j+1,n-j+1}+\gen{\dl}).
\end{align*}
However
\begin{align*}
	e_{n-i+1,n-i+1}\cdot e_{n-j+1,n-j+1}\in \gen{\dl}&\iff (i,j)\not\in\{(1,1),(1,n),(n,1),(n,n)\}
	\iff e_{ii}\cdot e_{jj}\in \gen{\dl}.
\end{align*}
Thus, we again conclude that  \cref{e_ii*e_jj-in-<dl>-iff-e_ii.e_jj-in-<dl>} holds for all $1\le i,j\le n$ showing that the structures \cref{cdot-of-Poisson-type,cdot-sum-of-Poisson-type-with-fixed} cannot be isomorphic.

Conversely, it is directly verified that \cref{e_11.e_11=e_nn.e_nn=-e_11.e_nn=-e_nn.e_11=e_1n} is a transposed Poisson algebra structure on $T_n(F)$, and it is orthogonal to any transposed Poisson algebra structure of Poisson type on $T_n(F)$.
\end{proof}

The following \cref{descr-TP-on-T_2(F)} describes transposed Poisson structures on $T_2(F)$, and its proof is similar to that of \cref{descr-TP-on-T_n(F)} with the only difference that as a result we will have $3$ types of the structures, and inside each of the $3$ types the structures can be fully classified up to isomorphism. On the other hand, $T_2(F)$ is isomorphic to the $3$-dimensional Lie algebra $\mathfrak{g}_2^0$ from~\cite{bfk23} whose transposed Poisson structures were fully described in \cite[Proposition 14]{bfk23}. By this reason, we have chosen to omit the proof of \cref{descr-TP-on-T_2(F)} and indicate on the left of each of the found structures its isomorphic version from~\cite{bfk23}.

\begin{thrm}\label{descr-TP-on-T_2(F)}
	Let $\ch(F)=0$. Then any transposed Poisson algebra structure on $T_2(F)$ is isomorphic to exactly one of the following structures:
	\begin{enumerate}
		\item\label{cdot-of-Poisson-type-n=2} transposed Poisson algebra structure of Poisson type:
            \begin{enumerate}
                \item\label{trivial-TP} $\mathrm{T}_{09}^{0,0}:$ the trivial one;
                \item\label{e_11-cdot-e_11=dl} $\mathrm{T}_{17}^0:\ e_{11}\cdot e_{11}=\dl$;
                \item\label{e_ii-cdot-e_jj=(-1)^(i+j)dl} $\mathrm{T}_{10}^0:\ e_{11}\cdot e_{11}=-e_{11}\cdot e_{22}=e_{22}\cdot e_{22}=\dl$;
            \end{enumerate}
		\item\label{cdot-sum-of-Poisson-type-with-fixed-n=2} the orthogonal sum of the transposed Poisson algebra structure
		\begin{align*}
			e_{11}\cdot e_{11}=-e_{11}\cdot e_{22}=e_{22}\cdot e_{22}=e_{12}
		\end{align*}
  with a transposed Poisson algebra structure of Poisson type, which results in the structures:
            \begin{enumerate}
              \item\label{trivial-TP+(-1)^(i+j)e_12} $\mathrm{T}_{16}:\ e_{11}\cdot e_{11}=-e_{11}\cdot e_{22}=e_{22}\cdot e_{22}=e_{12}$;
             \item\label{e_11-cdot-e_11=dl+(-1)^(i+j)e_12} $\mathrm{T}_{18}:\ e_{11}\cdot e_{11}=e_{12}+\dl$, $-e_{11}\cdot e_{22}=e_{22}\cdot e_{22}=e_{12}$;
              \item\label{e_ii-cdot-e_jj=(-1)^(i+j)dl+(-1)^(i+j)e_12} $\mathrm{T}_{11}^0:\ e_{11}\cdot e_{11}=-e_{11}\cdot e_{22}=e_{22}\cdot e_{22}=e_{12}+\dl$;
            \end{enumerate}
		\item\label{cdot-sum-of-Poisson-type-with-family-n=2} the (non-orthogonal) sum of a transposed Poisson algebra structure of the family
		\begin{align*}
			e_{11}\cdot e_{11}=ce_{11},
			e_{11}\cdot e_{12}=-e_{12}\cdot e_{22}=ce_{12},
                e_{11}\cdot e_{22}=-e_{22}\cdot e_{22}=ce_{22},\ c\ne 0,
		\end{align*}
	 with a transposed Poisson algebra structure of Poisson type, which results in the structures:
        \begin{enumerate}
            \item\label{TP-c-family} $\mathrm{T}_{17}^{-c}:\ e_{11}\cdot e_{11}=ce_{11},\ e_{11}\cdot e_{12}=-e_{12}\cdot e_{22}=ce_{12},\ e_{11}\cdot e_{22}=-e_{22}\cdot e_{22}=ce_{22}$;
            \item\label{TP-c-family+e_22^2=-c.dl} $\mathrm{T}_{09}^{0,-c}:\ e_{11}\cdot e_{11}=ce_{11},\ e_{11}\cdot e_{12}=-e_{12}\cdot e_{22}=ce_{12},\ e_{11}\cdot e_{22}=ce_{22},\ e_{22}\cdot e_{22}=-c(e_{22}+\dl)$;
            \item\label{TP-c-family+e_11e_22=-c.dl_e_22^2=c.dl} $\mathrm{T}_{19}^{-c}:\ e_{11}\cdot e_{11}=ce_{11},\ e_{11}\cdot e_{12}=-e_{12}\cdot e_{22}=ce_{12},\ -e_{11}\cdot e_{22}=e_{22}\cdot e_{22}=ce_{11}$.
        \end{enumerate}
	\end{enumerate} 
\end{thrm}

	\section{Transposed Poisson structures on the full matrix Lie algebra}\label{sec-TP-on-M_n}
In this short section we describe transposed Poisson structures on the full matrix algebra $M_n(F)$. As above, we denote by $e_{ij}$, $1\le i,j\le n$, the matrix units and by $\dl$ the identity matrix. Recall that 
\begin{align*}
    Z(M_n(F))&=\gen{\dl},\\
    [M_n(F),M_n(F)]&=\mathfrak{sl}_n(F)=\{a\in M_n(F)\mid \tr(a)=0\}\\
    &=\gen{e_{ij}\mid 1\le i\ne j\le n}\oplus\gen{e_{11}-e_{ii}\mid 1<i\le n}.     
\end{align*}
Denote by $\Dl(M_n(F))$ the space of $\frac 12$-derivations of the Lie algebra $(M_n(F),[\cdot,\cdot])$. The linear map $\gm:M_n(F)\to M_n(F)$ given by
\begin{align*}
	\gm(e_{ij})=
	\begin{cases}
		\dl, & i=j,\\
		0, & i\ne j,
	\end{cases}
\end{align*}
belongs to $\Dl(M_n(F))$ as a linear map $M_n(F)\to Z(M_n(F))$ annihilating $[M_n(F),M_n(F)]$.

 \begin{prop}\label{descr-Dl(M_n(F))}
     Let $n>1$. Then $\Dl(M_n(F))=\gen{\id,\gm}$.
 \end{prop}
 \begin{proof}
     It is easy to see that, as a Lie algebra, $M_n(F)$ is the direct sum $[M_n(F),M_n(F)]\oplus Z(M_n(F))$ Namely, every $a\in M_n(F)$ decomposes uniquely as 
     \begin{align*}
        a=\left(a-\frac {\tr(a)}n\dl\right)+\frac {\tr(a)}n\dl.    
     \end{align*}
     Since $\frac 12$-derivations of a Lie algebra $\LL$ leave $[\LL,\LL]$ and $Z(\LL)$ invariant, we have
     \begin{align*}
         \Dl(M_n(F))=\Dl([M_n(F),M_n(F)])\oplus\Dl(Z(M_n(F))).
     \end{align*}
     By \cite[Corollary 3]{fil1} any $\vf\in\Dl([M_n(F),M_n(F)])$ is trivial. Since $\dim(Z(M_n(F)))=1$, any $\vf\in\Dl(Z(M_n(F)))$ is trivial as well. Hence, given $\vf\in\Dl(M_n(F))$, there exist $k_1,k_2\in F$ such that $\vf=k_1\cdot\id_{[M_n(F),M_n(F)]}+k_2\cdot\id_{Z(M_n(F))}$. Clearly,
     \begin{align*}
         \vf(e_{ij})&=k_1e_{ij},\text{ if }i\ne j,\\
         \vf(e_{ii})&=k_1\left(e_{ii}-\frac 1n\dl\right)+\frac{k_2}n\dl=k_1e_{ii}+\frac{k_2-k_1}n\dl.
     \end{align*}
It follows that $\vf=k_1\cdot\id_{M_n(F)}+\frac{k_2-k_1}n\cdot\gm$. Thus, $\Dl(M_n(F))\sst\gen{\id,\gm}$. The converse inclusion is trivial.
 \end{proof}

 \begin{thrm}\label{descr-TP-on-M_n(F)}
	Let $\ch(F)=0$ and $n>1$. Then, up to isomorphism, there is only one non-trivial transposed Poisson algebra structure on $M_n(F)$. It is given by
 \begin{align}\label{e_ii-cdot-e_jj=dl}
     e_{ii}\cdot e_{jj}=\dl,\ 1\le i,j\le n,
 \end{align}
 and it is of Poisson type.
\end{thrm}
\begin{proof}
    Let $\cdot$ be a transposed Poisson algebra structure on $M_n(F)$. By \cref{descr-Dl(M_n(F)),glavlem} for all $1\le i,j\le n$ there exist $x_{ij}$ and $y_{ij}\in F$, such that
		\begin{align*}
			e_{ij}\cdot e_{kl}=x_{ij}e_{kl}+y_{ij}\gm(e_{kl})=
			\begin{cases}
				x_{ij}e_{kl}, & k\ne l,\\
				x_{ij}e_{kk}+y_{ij}\dl, & k=l.
			\end{cases}
		\end{align*}
		Let $i\ne j$ and $1\le k\le n$. Then $e_{kk}\cdot e_{ij}=x_{kk}e_{ij}$, while $e_{ij}\cdot e_{kk}=x_{ij}e_{kk}+y_{ij}\dl$, whence
		\begin{align*}
			x_{kk}=0\text{ and }x_{ij}=y_{ij}=0\text{ for }i\ne j\text{ and }1\le k\le n.
		\end{align*}
		It follows from $e_{ii}\cdot e_{jj}=y_{ii}\dl$ and $e_{jj}\cdot e_{ii}=y_{jj}\dl$ that
		\begin{align*}
			y_{ii}=y_{jj}\text{ for }1\le i,j\le n.
		\end{align*}
		Thus, denoting $y_{ii}$ by $c$, we have the only (possibly) non-zero products 
		\begin{align*}
			e_{ii}\cdot e_{jj}=c\dl,\ 1\le i,j\le n.
		\end{align*}
		If $c=0$, then $\cdot$ is trivial. Otherwise, choosing the automorphism $\phi$ of $(M_n(F),[\cdot,\cdot])$ given by
		\begin{align*}
			\phi(e_{ij})&=e_{ij},\text{ if }i\ne j,\\
			\phi(e_{ii})&=e_{ii}-\frac 1n\dl+\frac cn\dl=e_{ii}+\frac {c-1}n\dl,
		\end{align*}
		we obtain the isomorphic structure \cref{e_ii-cdot-e_jj=dl}. 
		
		It remains to see that \cref{e_ii-cdot-e_jj=dl} is of Poisson type. For all $a,b\in M_n(F)$ we have
		\begin{align*}
			a\cdot b=\sum_{i,j=1}^n a(i,i)a(j,j)e_{ii}\cdot e_{jj}=\sum_{i,j=1}^n a(i,i)a(j,j)\dl=\tr(a)\tr(b)\dl=\frac{\tr(a)}n\dl\cdot\frac{\tr(b)}n\dl.
		\end{align*}
		Hence, $\cdot$ is the extension by zero of a product on the complement $Z(M_n(F))$ of $[M_n(F),M_n(F)]$ in $M_n(F)$ with values in $Z(M_n(F))$.
\end{proof}

\begin{rem}\label{ext-dl.dl=dl}
    The structure \cref{e_ii-cdot-e_jj=dl} is isomorphic to the extension by zero of the product $\dl\cdot \dl=\dl$.
\end{rem}

	
	




\end{document}